\documentclass[12pt]{amsart}
\usepackage{xcolor}
\usepackage[a4paper,asymmetric]{geometry}
\usepackage{latexsym}
\usepackage{epsfig}
\usepackage{amsthm}
\usepackage{amssymb}
\usepackage{amsfonts}
\usepackage{amsmath}
\usepackage{epstopdf}
\newtheorem{theorem}{Theorem}[section]
\newtheorem{thm}[theorem]{Theorem}

\newtheorem{lem}[theorem]{Lemma}
\newtheorem{proposition}[theorem]{Proposition}
\newtheorem{prop}[theorem]{Proposition}
\newtheorem{corollary}[theorem]{Corollary}
\newtheorem{assumption}[theorem]{Assumption}
\theoremstyle{definition}

\newtheorem{defn}[theorem]{Definition}
\newtheorem{ex}[theorem]{Example}
\theoremstyle{remark}

\newtheorem{rem}[theorem]{Remark}
\numberwithin{equation}{section}

 \DeclareMathAlphabet{\mathpzc}{OT1}{pzc}{m}{it}
 
 \newcommand{\E}{\mathbb{E}}            
 
 \newcommand{\e}{\varepsilon}
 \newcommand{\p}{\partial}

 \newcommand{\N}{\mathbb{N}}
 \newcommand{\R}{\mathbb{R}}
 
 \newcommand{\PP}{\mathbb{P}}
 \newcommand{\HH}{\mathbb{H}}
 \newcommand{\mcl}{\mathcal}
 
 \newcommand{\Be}{\begin{equation}}
 \newcommand{\Ee}{\end{equation}}
 \newcommand{\Bs}{\begin{split}}
 \newcommand{\Es}{\end{split}}
  \newcommand{\Bes}{\begin{equation*}}
 \newcommand{\Ees}{\end{equation*}}
 \newcommand{\BT}{\begin{thm}}
 \newcommand{\ET}{\end{thm}}
 \newcommand{\Bp}{\begin{proof}}
 \newcommand{\Ep}{\end{proof}}
 \newcommand{\BL}{\begin{lem}}
 \newcommand{\EL}{\end{lem}}
 \newcommand{\BP}{\begin{proposition}}
 \newcommand{\EP}{\end{proposition}}
 \newcommand{\BC}{\begin{corollary}}
 \newcommand{\EC}{\end{corollary}}
 \newcommand{\BR}{\begin{rem}}
 \newcommand{\ER}{\end{rem}}
 \newcommand{\BD}{\begin{defn}}
 \newcommand{\ED}{\end{defn}}
 \newcommand{\BI}{\begin{itemize}}
 \newcommand{\EI}{\end{itemize}}
 \newcommand{\eqn}{equation}
 \newcommand{\tl}{\tilde}

 \newcommand{\bg}{\big}

 \newcommand{\dif}{{\rm d}}

\begin{document}
\title
[Exponential mixing for SPDEs with highly degenerate L\'evy noise]
{Exponential mixing for SPDEs driven by highly degenerate L\'evy noises}

\author[X. Sun Y. Xie]{Xiaobin {\sc Sun}\ ,
\ Yingchao {\sc Xie} }\thanks{X. Sun and Y. Xie are supported by the National
Natural Science Foundation of China (11271169) and the
Project Funded by the Priority Academic Program Development of
Jiangsu Higher Education Institutions.  Lihu Xu is supported by the following grants: Science and Technology Development Fund Macao S.A.R FDCT  049/2014/A1,  MYRG2015-00021-FST}
\address{School of Mathematics and Statistics\\  Jiangsu Normal University\\  Xuzhou 221116, China}
\email{xbsun@jsnu.edu.cn,   ycxie@jsnu.edu.cn }

\author[L. Xu]{Lihu Xu}
\address{Department of Mathematics,
Faculty of Science and Technology,
University of Macau, E11
Avenida da Universidade, Taipa,
Macau, China}
\email{lihuxu@umac.mo}


\maketitle

\begin{abstract} \label{abstract}
By a coupling method, we prove that a family of stochastic partial differential equations (SPDEs) driven by
highly degenerate pure jump L\'evy noises are exponential mixing. These pure jump L\'evy noises include
$\alpha$-stable process with $\alpha \in (0,2)$.

\end{abstract}

\section{Introduction}
Let $\mathbb{H}$ be a Hilbert space with a complete orthonormal basis $\{e_k\}_{k\in\mathbb{N}}$. Let $A:\mcl D(A)\subset\mathbb{H}\to\mathbb{H}$ be a self-adjoint operator such that
$$Ae_n=-\lambda_n e_n, \ \ \  \ n \in \N$$
where $0<\lambda_1 \le \lambda_2 \le ...$ and $\lim_{n \rightarrow \infty}\lambda_n=\infty$.
We are concerned with the following stochastic PDEs:
\Be \label{e:MaiSPDE}
\dif X_t=[A X_t+F(X_t)] \dif t+\dif L_t
\Ee
where $F: \mathbb{H} \rightarrow \mathbb{H}$ is bounded and Lipschitz and $L_t$ is a $D$-dimensional pure jump L\'evy process on the subspace ${\rm span}\{e_1,...,e_D\}$ (see Assumption \ref{a:} below).
Before giving the main theorem, let us first point out that the problem \eqref{e:MaiSPDE} is well-posed.
By the same method as in \cite[Sect. 5.1]{PZ09}, we can show that
for any initial data $x \in \HH$, equation \eqref{e:MaiSPDE} has a unique mild solution
$(X^x(t))_{t \ge 0}$ with Markov property as follows:
\
\begin{\eqn} \label{e:MilSol}
X^x(t)=e^{At}x+\int_0^t e^{A(t-s)} F(X^x(s))\dif s+\int_0^t e^{A(t-s)}
\dif L_s.
\end{\eqn}
Moreover, $(X^x(t))_{t \ge 0}$ has a c$\grave{a}$dl$\grave{a}$g version in $\HH$ since $L_t$ is finite dimension.
Stochastic PDEs driven by non-degenerate L\'evy noise have been intensively studied in the past decades, see
\cite{AWZ98,DPZ96,DonXie11,FuXi09, Ku09, MaRo09, Mas, PeZa07, PZ09} and the references therein.

Eq. \eqref{e:MaiSPDE} is a highly degenerate stochastic partial differential equations (SPDEs) with L\'evy type noises.
As $L_t$ is highly degenerate
Wiener type or kick type noises, its ergodicity and related problems have been intensively studied recently, see \cite{Hai02, HM06, HaiMat11, Od07, Od09}
for Wiener type noises and \cite{KS-2001,KukShi02,Shi04,nersesyan-2008} for kick type noises. When $L_t$ is highly degenerate
L\'evy noises, to our knowledge, there seem no ergodicity results. One aim of our paper is to partly fill in this gap.

The third author of this paper studied in \cite{Xu14} a 2d degenerate SDE driven by 1d L\'evy noises, as the dissipation in the direction not driven by noises is sufficiently strong, the stochastic system is exponentially mixing. This paper will prove the same result for highly degenerate SPDE \eqref{e:MaiSPDE}, and adopt some notations and auxiliary lemmas in \cite{Xu14} for readers' convenience. We shall use a similar approach as in \cite{Xu14}
to proving exponential ergodicity, but we have to conquer some difficulties due to the infinite dimension (see Sections \ref{s:Ex} and \ref{s:ConCou} below). Moreover, our new coupling construction is much more involved and the proof is simplified with a different strategy.

In section \ref{s:Ex}, we give an example similar to Example 2.9 in \cite{PSXZ11}. The latter example shows that
a one dimensional SPDE driven by non-degenerate $\alpha$-stable noise with $\alpha \in (1,2)$ is exponential mixing, while
the latter one in this paper indicates that as $\alpha$-stable noise is highly degenerate, the SPDE can be $d$-dimensional for all
$d \ge 1$ and $\alpha$ can be in $(0,2)$.
The two restrictions $d=1$ and $\alpha \in (1,2)$ in Example 2.9 of \cite{PSXZ11} are hard to be removed due to the limitation of the Harris' approach to ergodicity (see also \cite{DoXuZh11,DoXuZh14} for some other examples). So, from these two examples, we can see that our coupling approach has big advantage for studying the ergodicity of stochastic L\'evy type systems .

The structure of the paper is as follows. The remaining part of this section introduces the notations and gives the main theorem. Section 2
gives an example to which our main theorem is applicable and shows that L\'evy type noises include $\alpha$-stable noise with $\alpha \in (0,2)$.
We construct a coupling Markov process in the 3rd section and prove its properties which are important for estimating the stopping times in Section 4.
In the last section, we prove the main theorem with a strategy given at the beginning.
\subsection{Some preliminary of L\'evy process (\cite{Ber96}) and the assumptions}
Let $L_t$ be a D-dimensional L\'evy process with L\'evy measure $\nu$, denote
$$
\Delta L_t=L_t-L_{t-}.
$$
For any $K>0$, define
$$
\Gamma_{K}:=\{y\in\mathbb{R}^D: |y|\geq K\},\quad \gamma_{K}:=\nu(\Gamma_{K}).
$$
Note that $\gamma_{K}$ is a decreasing function of $K$ and $\gamma_{K}<\infty$ for any $K>0$.

For any $T>0$, define
\begin{equation}\label{stopping time}
\tau:=\inf\{t> T: |\Delta L_t|\geq K\}.
\end{equation}
Then $\tau$ is a stopping time with density
$$
\gamma_{K}\exp\{-\gamma_{K}(t-T)\}1_{\{t>T\}}.
$$
Define $\tau_0:=0$ and
$$
\tau_k:=\inf\left\{t>\tau_{k-1}+T: |\Delta L(t)| \ge K \right\} \ \ \
{\rm for \ all \ } k \ge 1.
$$
It is easy to see that $\{\tau_k\}_{k \ge 0}$ are a sequence of stopping times such that
\begin{equation} \label{e:TauKInd}
\{\tau_{k}-\tau_{k-1}\}_{k \ge 1} {\rm \ are \ independent \ and \ have \ the \ same \ density \ as}
\ \tau.
\end{equation}


\vskip 2mm

\begin{assumption}
\label{a:}
We assume that
\begin{itemize}
\item [(A1)] $\sup_{0 \le t<\infty} \E \left|\int_0^t e^{A (t-s)} \dif L_s\right|^p<\infty$ for
$p \in (0, \alpha)$ with some $\alpha>0$.
\item [(A2)] For some $K>0$, $\nu_K$ has a density $p_K$ such that for all $x_1, x_2 \in \R^D$,
\Bes
\int_{\R^D} |p_K(z-x_1)-p_K(z-x_2)| dz \le \beta_1 |x_1-x_2|^{\beta_2},
\Ees
where $\beta_1, \beta_2>0$ are constants only depending on $K$.
\item [(A3)] There exist some $M>0$ and some $\beta_0=\beta_0(K,M) \in (0,2)$ such that if $|x_1|+|x_2| \le M$,
\Bes
\int_{\R^D} |p_K(z-x_1)-p_K(z-x_2)| dz \le \beta_0.
\Ees
\item [(A4)] $\gamma_{K} \ge 2 \beta_2 \|F\|_{{\rm Lip}}$.
\end{itemize}
\end{assumption}

\begin{rem}
The number ``$2$''in ``$\gamma_{K} \ge 2 \beta_2 \|F\|_{{\rm Lip}}$'' of (A4) can be replaced by any number $c>1$.
We choose the special ``$2$'' to make the computation in sequel more simple. The number $M$ will be chosen in Theorem \ref{l:TauPro}.
\end{rem}

\subsection{Some notations for the further use}
Denote by $B_b(\HH)$ the Banach space of bounded Borel-measurable functions $f:\HH \rightarrow \R$ with the norm
$$
\|f\|_{0}:=\sup_{x \in \HH} |f(x)|.
$$
Denote by $L_b(\HH)$ the Banach space of global Lipschitz bounded functions $f: \HH \rightarrow \R$ with the norm
$$
\|f\|_{1}:=\|f\|_0+\|f\|_{{\rm Lip}}.
$$
where
$\|f\|_{{\rm Lip}}:=\sup_{x \neq y}\frac{|f(x)-f(y)|}{\|x-y\|_{\HH}}.$

\smallskip
Let $\mcl B(\HH)$ be the Borel $\sigma$-algebra on $\HH$ and let~$\mcl P(\HH)$ be
the set of probability measures on $(\HH,\mcl B(\HH))$. Recall that the total variation distance
between two measures $\mu_1, \mu_2 \in \mcl P(\HH)$ is defined by
$$
\|\mu_1-\mu_2\|_{\rm TV}
:=\frac{1}{2}\sup_{\stackrel{f \in B_b(\HH)}{\|f\|_0=1}} |\mu_1(f)-\mu_2(f)|
=\sup_{\Gamma\in\mcl B(\HH)}|\mu_1(\Gamma)-\mu_2(\Gamma)|.
$$
Given a random variable
$X$, we shall use
$$\mcl L(X) {\rm \ to \ denote \ the \ distribution \ of \ X.}$$

Let $\Pi$ be the orthogonal projection from $\HH$ to the subspace
${\rm span}\{e_1,...,e_D\}$. For any $x \in \HH$, define
\
\Be \label{e:Pi}
x_1=\Pi x, \ \ \  \ x_2=(Id-\Pi) x.
\Ee
For the further use, we denote
$$\HH_1:={\rm span}\{e_1,...,e_D\}, \ \ \ \HH_2:={\rm span}\{e_{D+1},e_{D+2},...,\}.$$
Then Eq. \eqref{e:MaiSPDE} can be written as
\
\Be  \label{e:X1Eqn}
\dif X_1(t)=[A X_1(t)+F_1(X(t))]\dif t+\dif L_t,
\Ee
\Be  \label{e:X2Eqn}
\dif X_2(t)=[A X_2(t)+F_2(X(t))]\dif t.
\Ee

Let us denote by $(P_t)_{t\geq 0}$ the Markov semigroup associated with Eq. ~\eqref{e:MaiSPDE}, i.e.
\begin{\eqn*}
P_t f(x):=\E\left[f(X^x(t))\right], \quad f \in B_b(\HH),
\end{\eqn*}
and by $(P^{*}_t)_{t \geq 0}$ the dual semigroup acting on $\mcl P(\HH)$.

\subsection{Main result}

Our \emph{main result} is the following ergodic theorem and it will be proven in the last section.

\begin{thm}\label{t:MaiThm}
Under Assumption \ref{a:}, if $D$ is sufficiently large,
then the system \eqref{e:MaiSPDE} is ergodic and exponentially mixing under the weak topology of
$\mcl P(\HH)$. More precisely, there exists a unique invariant measure
$\mu \in \mcl P(\HH)$ so that for any $p \in (0, \alpha)$ and
any measure $\tl \mu \in \mcl P(\HH)$ with finite $p^{\rm th}$ moment, we have
\begin{equation} \label{2.5}
\left|\int_{\HH} f \dif P^{*}_t \tl \mu-\int_{\HH} f \dif \mu\right| \leq C e^{-ct}\|f\|_1\left(1+\int_{\HH} |x|^{p}
\tl \mu(dx)\right) \ \ \ \ \forall \ f \in L_b(\HH),
\end{equation}
where $C,c$ depend on $p, K, \|F\|_{{\rm Lip}}, \|F\|_0, \lambda_1, \lambda_{D+1}$.
\end{thm}

\section{Examples and some preliminary estimates for the solution of Eq. \eqref{e:MaiSPDE}} \label{s:Ex}

\subsection{Some concrete examples for Eq. \eqref{e:MaiSPDE}}
We first claim that
\
\begin{proposition}  \label{p:StaCon}
$D$-dimensional rotationally symmetric $\alpha$-stable process $L_t$, with $0<\alpha<2$ and $D \in \N$,  satisfies
Assumption \ref{a:}.
\end{proposition}
Before proving the proposition, we give an example below to which the assumptions of the paper applies, c.f. Example 2.9 in \cite{PSXZ11}.
\begin{ex} \label{comp}
Consider the following stochastic semilinear equation on $\mcl D=[0,\pi]^d$
with $d \geq 1$ and the Dirichlet boundary condition:
\Be
\label{e:SemLin}
\begin{cases}
\dif X(t,\xi)=[\Delta X(t,\xi)+F(X(t,\xi))]\dif t+\dif L_t(\xi), \\
X(0,\xi)=x(\xi), \\
X(t,\xi)=0, \ \ \xi \in \p \mcl D,
\end{cases}
\Ee
where $F$ is bounded Lipschitz, $L_t$ is a $D$-dimensional rotationally symmetric $\alpha$-stable processes with $\alpha \in (0,2)$ to be further specified below.

It is well known that $\Delta$ with Dirichlet boundary condition has the following
eigenfunctions
$$e_k(\xi)=\left(\frac{2}{\pi}\right)^{\frac d2}
\sin(k_1 \xi_1) \cdots \sin(k_d \xi_d), \quad k \in \N^d, \ \xi \in \mcl D.
$$
It is easy to see that $\Delta e_k=-|k|^2 e_k$,
i.e. $\gamma_k=|k|^2= k_1^2 + \ldots + k_d^2$ for all $k \in \N^d$.
We study the dynamics defined by~\eqref{e:SemLin} in the Hilbert space $\HH= L^2(\mcl D)$  with orthonormal basis $\{e_k\}_{k \in \N^d}$.


$L_t$ is a $D$-dimensional symmetric $\alpha$-stable processes on the subspace ${\rm span}\{e_1,...,e_D\}$. From our main result Theorem \ref{t:MaiThm}, for all $\alpha \in (0,2)$, as $D$ is sufficiently large, the stochastic system \eqref{e:SemLin} converges to equilibrium measure
exponentially fast.

Let us roughly compare our example with Example 2.9 of \cite{PSXZ11}, which has the same form as Eq. \eqref{e:SemLin} but with $d=1$ and $\alpha \in (1,2)$. The two restrictions $d=1$ and $\alpha \in (1,2)$ are hard to be removed due to the limitation of the Harris' approach to ergodicity. To use Harris' ergodicity theorem, one has to prove the strong Feller property which is true for Example 2.9 of  \cite{PSXZ11} when $d=1$ and $L_t$ is non-degenerate $\alpha$-stable noises with $\alpha>1$. So, from these two examples, we can see the advantage of our coupling approach to the ergodicity.
 \end{ex}

\begin{proof} [{\bf Proof of Proposition \ref{p:StaCon}}]
 Recall that $D$-dimensional rotationally symmetric $\alpha$-stable process $L_t$ has the following representation:
 $$L_t=W_{S_t}$$
 where $W_t=(W^1_t,\ldots, W^{D}_t)$ be a D-dimensional standard Brownian motion and $S_t$ is an $\alpha/2$-stable subordinator
 independent of $W_t$.

Let $\E^{\mathbb S}, \E^{\mathbb W}$ denote the partial integrations with respect to $S$ and $W$ respectively, we have
 \Be
 \begin{split}
 \E \left|\int_0^t  e^{A(t-s)} \dif L_s\right|^p&=\E^{\mathbb S}\left[\E^{\mathbb W}\left(\int_0^t  e^{A(t-s)} \dif W_{\ell_s}\right)^p \big|_{\ell.=S.}\right] \\
 & \le \E^{\mathbb S}\left[\left(\int_0^t  \sum_{i=1}^D e^{-2\lambda_i (t-s)} \dif \ell_s\right)^{p/2} \big|_{\ell.=S.}\right] \\
 &=\E\left(\int_0^t  \sum_{i=1}^D e^{-2\lambda_i (t-s)} \dif S_s\right)^{p/2} \\
 &\le D^{p/2} \E\left(\int_0^t e^{-2\lambda_1 (t-s)} \dif S_s\right)^{p/2}. \\
 \end{split}
 \Ee
Then,
\Bes
\begin{split}
&\E\left(\int_0^t e^{-2 \lambda_1(t-s)} \dif S_s\right)^{p/2}  \le \E \left(\sum_{k=0}^{[t]}\int_{k}^{k+1} e^{-2 \lambda_1 (t-s)} \dif S_s\right)^{p/2}  \\
& \le \sum_{k=0}^{[t]} \E\left(\int_{k}^{k+1} e^{-2 \lambda_1(t-s)} \dif S_s\right)^{p/2}
 \le \sum_{k=0}^{[t]}  e^{-p\lambda_1(t-k-1)} \E\left(S_{k+1}-S_k\right)^{p/2} \\
& =\E S^{p/2}_1 \sum_{k=0}^{[t]}  e^{-p\lambda_1(t-k-1)}
 \le \E S^{p/2}_1 e^{2 p \lambda_1} \sum_{k=0}^{\infty}  e^{-p k \lambda_1}.
\end{split}
\Ees
Thus, (A1) is immediately verified from the above estimates.

Let us now verify that (A2) holds for all $K>0$ (this is obviously stronger than (A2) itself).
For any $K>0$, the density of $\nu_K$ is
$$
p_K(z)=\frac{C}{|z|^{D+\alpha}}1_{\{|z|\geq K\}},
$$
where $C=C_{\alpha,D,K}$ depends on $\alpha, D, K$.

Without loss of generality, we assume $x_1=0$, $x_2=(x, 0,\ldots, 0)$ with $x \in \R$.
Denote $f(x):=\int_{\mathbb{R}^D}|p_K(z)-p_{K}(z-x_2)|dz$, note $f(0)=0$. We will show
\begin{eqnarray}
\sup_{|x| \leq K}|f'(x)|\leq C_K.\label{f}
\end{eqnarray}
On the other hand, as $|x|>K$, we have
$$f(x)\leq 2\leq\frac{2}{K}|x|=\frac{2}{K}|x_2|.$$
Combining the above two relations, we immediately get
\Be  \label{e:SimA2}
\int_{\mathbb{R}^D}|p_K(z)-p_{K}(z-x_2)|dz\leq \left(C_K\vee \frac{2}{K}\right)|x_2|.
\Ee
Hence, we verified that (A2) holds with $\beta_1=C_K\vee \frac{2}{K}$ and $\beta_2=1$.

Let now show (\ref{f}). We first divide $\R^D$ into several parts (for instance, see Figure 1 when $D=1$),
\begin{figure}[http]
\centering
\includegraphics[width=2.8in,height=2.2in]{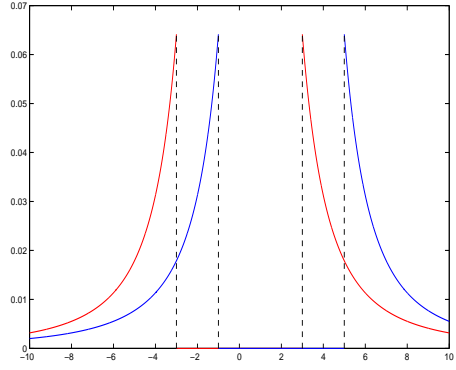}
\caption{The red and blue lines are the functions of $\frac{1}{|z|^{1+\alpha}}1_{\{|z|\geq K\}}$ and $\frac{1}{|z-2|^{1+\alpha}}1_{\{|z-2|\geq K\}}$ with $\alpha=\frac{3}{2}$ and $K=3$ respectively.}
\end{figure}

\begin{eqnarray*}
f(x)&=&\int_{\mathbb{R}^D}\left|\frac{C}{|z|^{D+\alpha}}1_{|z|\geq K}-\frac{C}{|z-x_2|^{D+\alpha}}1_{|z-x_2|\geq K}\right|dz\\
&=&2-4\int_{z_1\leq x/2, z^{2}_1+\cdots +z^{2}_D\geq K^2}\frac{C}{[(z_1-x)^2+z^{2}_2+\cdots +z^{2}_D]^{\frac{D+\alpha}{2}}}dz_1\cdots dz_D\\
&\hat=&2-4g(x).
\end{eqnarray*}
When $D=1$, we can easily get (\ref{f}) holds. The proof of $D=2$ or $D\geq 3$ are almost the same, so we only study the case of $D\geq 3$ for convenience, using the transformation of spherical coordinates, i.e.
\begin{eqnarray*}
z_1&=&r\cos\theta_1,\\
z_2&=&r\sin\theta_1\cos\theta_2,\\
&\ldots&\\
z_{D-1}&=&r\sin\theta_1\sin\theta_2\cdots\cos\theta_{D-1}, \\
z_{D}&=&r\sin\theta_1\sin\theta_2\cdots\sin\theta_{D-1},
\end{eqnarray*}
where $\theta_1,\ldots,\theta_{D-2}\in [0,\pi]$, $\theta_{D-1}\in [0, 2\pi]$, $r\geq 0$. Then, we have
\begin{eqnarray*}
g(x)&=&\int_{r\cos\theta_1\leq x/2, r\geq K}\frac{C r^{D-1}\sin^{D-2}\theta_1\sin^{D-3}\theta_2\cdots\sin\theta_{D-2}}{[(r\cos\theta_1-x)^2+r^2\sin^{2}\theta_1]^{\frac{D+\alpha}{2}}}d\theta_1\cdots d\theta_{D-1}dr\\
&=&C_{\alpha, D}\int_K^\infty \int^{\pi}_{arc\cos \frac{x}{2r}} \frac{r^{D-1} \sin^{D-2} \theta_1} {\left[(r \cos \theta_1-x)^2+r^2 \sin^2 \theta_1\right]^{\frac{D+\alpha}2}} dr d \theta_1.
\end{eqnarray*}
Take the derivative with respect to $x$, we get
\begin{eqnarray*}
g'(x)=\!\!\!\!\!\!\!\!\!\!&&C_{\alpha, D}\int_K^\infty  \frac{\left(1-\frac{x^2}{4r^2}\right)^{\frac{D-2}2}}{r^{\alpha+1}\sqrt{4r^2-x^2}}dr\\
&&+C_{\alpha, D}\int_K^\infty \int^{\pi}_{arc\cos \frac{x}{2r}} \frac{r^{D-1} \sin^{D-2} \theta_1 (r \cos \theta_1-x)} {\left[(r \cos \theta_1-x)^2+r^2 \sin^2 \theta_1\right]^{\frac{D+\alpha}2+1}} dr d \theta_1.
\end{eqnarray*}
Then, it is easy to see that
$$\sup_{|x| \le K}|g'(x)|<C_K,\quad \text{as} \quad |x|\leq K.$$
Hence, (\ref{f}) holds.

Since the supports of the functions $p_K(z-x_1)$ and $p_K(z-x_2)$ have overlaps, it holds that
\Bes
\begin{split}
\phi(x_1,x_2):&=\int_{\R^D} |p_K(z-x_1)-p_K(z-x_2)| dz \\
&<\int_{\R^D} p_K(z-x_1)dz+\int_{\R^D}p_K(z-x_2) dz=2.
\end{split}
\Ees
Since $\phi$ is a continuous function, for all $M>0$ there exists some $\beta_0 \in (0,2)$ depending
on $M$ and $K$ such that (A3) holds.

Since $L_{t}$ is $\alpha$-stable noise,
 $\gamma_K \rightarrow \infty$ as $K \downarrow 0$. Therefore,
(A4) is clearly true.
\end{proof}

\subsection{Some easy estimates about the solution}
In this subsection, we prove some easy estimates about the solution $X(t)$ of Eq.
\eqref{e:MaiSPDE}, which will play an essential role for estimating some
stopping times in the sections later.

\begin{lem} \label{l:SolEst}
The following statements hold:
\begin{enumerate}
\item For $x,y \in \HH$, $p \in (0, \alpha)$, we have
\Bes
\E \|X^x(t)\|_{\HH}^p \le (3^{p-1} \vee 1)  e^{-\lambda_1 pt} \|x\|_{\HH}^p+C \ \ \ \ \forall \ t \ge 0,
\Ees
\Bes
\E \|X^x(t)-X^y(t)\|_{\HH}^p \le (2^{p-1} \vee 1)  e^{-\lambda_1 pt} \|x-y\|_{\HH}^p+C \ \ \ \ \forall \ t \ge 0,
\Ees
where $a \vee b:=\max\{a,b\}$ for $a,b \in \R$ and $C$ depends on $p,\lambda_1,\|F\|_0$.
\item For $x, y \in \HH$, we have
\Bes \label{e:X1X2Bou}
\|X^x(t)-X^y(t)\|_{\HH}\le e^{t \|F\|_{\rm Lip}} \|x-y\|_{\HH},
\Ees
\Bes \label{e:X1X2HBou}
\|X^x_2(t)-X^y_2(t)\|_{\HH} \le \left(e^{-\lambda_{D+1} t}+\frac{\|F\|_{\rm Lip}}{\lambda_{D+1}+\|F\|_{\rm Lip}} \ e^{t \|F\|_{\rm Lip}}\right) \|x-y\|_{\HH},
\Ees
for all $t \ge 0$.
\end{enumerate}
\end{lem}

\begin{proof}
Denote
$$Z_A(t):=\int_0^t e^{A (t-s)} \dif L_s,$$
it is easy to see that $Z_A(t)$ is a D-dimensional stochastic process and
$$|Z_A(t)|=\|Z_A(t)\|_{\HH}.$$
By \eqref{e:MilSol} we have
\Bes
\begin{split}
\left\|X^x(t)\right\|_{\HH} & \le \left\|e^{At} x\right\|_{\HH}+\left\|\int_0^t e^{A(t-s)} F(X^x(s)) \dif s\right\|_{\HH}+\left|Z_A(t)\right|\\
& \le e^{-\lambda_1 t} \|x\|_{\HH}+\int_0^t e^{-\lambda_1(t-s)} \dif s \|F\|_0+\left|Z_A(t)\right|,
\end{split}
\Ees
and
\Bes
\begin{split}
\left\|X^x(t)-X^y(t)\right\|_{\HH} & \le \left\|e^{At} (x-y)\right\|_{\HH}+\left\|\int_0^t e^{A(t-s)} [F(X^x(s))-F(X^y(s))] \dif s\right\|_{\HH} \\
& \le e^{-\lambda_1 t} \|x-y\|_{\HH}+2\int_0^t e^{-\lambda_1(t-s)} \dif s \|F\|_0.
\end{split}
\Ees
The first statement follows from the above inequality and (A1) of Assumption \ref{a:}.
\vskip 1.5mm

Let us now prove the second statement. It is easy to have
\Be \label{e:Xx-Xy}
X^x(t)-X^y(t)=e^{At}(x-y)+\int_0^t e^{A(t-s)} \left[F(X^x(s))-F(X^y(s))\right] \dif s,
\Ee
which implies
\Bes
\|X^x(t)-X^y(t)\|_{\HH} \le \|x-y\|_{\HH}+\int_0^t  \|F\|_{\rm Lip} \|X^x(s)-X^y(s)\|_{\HH} \dif s.
\Ees
From this we immediately get the first inequality by Gronwall's inequality. It follows from \eqref{e:Xx-Xy} and the first inequality that
\Bes
\begin{split}
\|X^x_2(t)-X^y_2(t)\|_{\HH} & \le \|e^{A t}(x_2-y_2)\|_{\HH}+\int_0^t \|e^{-A(t-s)} F_2(X^x(s))-F_2(X^y(s))\|_{\HH}ds \\
& \le e^{-\lambda_{D+1} t} \|x-y\|_{\HH}+\int_0^t e^{-\lambda_{D+1}(t-s)} \|F\|_{\rm Lip} \|X^x(s))-X^y(s)\|_{\HH} ds \\
& \le e^{-\lambda_{D+1} t} \|x-y\|_{\HH}+\int_0^t e^{-\lambda_{D+1}(t-s)} \|F\|_{\rm Lip} \ e^{s \|F\|_{\rm Lip}} \|x-y\|_{\HH} ds.
\end{split}
\Ees
This immediately implies the second inequality.
\end{proof}

\section{Construction of the coupling} \label{s:ConCou}
In this section, let us first construct a coupling Markov process which will play an essential
role for proving our ergodicity result, and then prove a preliminary estimate about this coupling.


\subsection{Construction of the coupling}
\begin{lem}  \label{l:CouLem1}
Let $X^x(t)$ and $X^y(t)$ be the solutions to the equation \eqref{e:MaiSPDE} for any given $x$ and $y$ respectively. Let $\tau$ be the stopping time defined by \eqref{stopping time}.
Then, we have a probability space $(\tl \Omega, \tl {\mcl F}, \tl \PP)$,  not depending on $x$ and $y$, on which there exist a random time $\tl \tau$ and a coupling Markov process
$S^{x,y}(t)=(S^x(t),S^y(t))_{0 \le t \le \tau}$ on $(\tl \Omega, \tl {\mcl F}, \tl \PP)$ such that
\begin{itemize}
\item[(1)] $\tl \tau$, not depending on $x$ and $y$, has the same distribution as $\tau$;
\item[(2)] $\mcl L((S^x(t))_{0 \le t \le \tl \tau})=\mcl L((X^x(t))_{0 \le t \le \tau})$ and $\mcl L((S^y(t))_{0 \le t \le \tl \tau})=\mcl L((X^y(t))_{0 \le t \le \tau})$;
\item[(3)] the following equality holds:
$$\tl \PP \bigg(S_1^x(\tl \tau) \ne S_1^y(\tl \tau)\big|S_1^x(\tl \tau-), S_1^y(\tl \tau-)\bigg)= \frac{1}{2}\int_{\R^D} |p_{K}(z-\hat x_1)-p_{K}(z-\hat y_1)| d z,$$
where $\hat x_1=S_1^x(\tl \tau-)$ and $\hat y_1=S_1^y(\tl \tau-)$, and $S_1^x(\tl \tau-)$ and $S_1^y(\tl \tau-)$
are defined by the practice \eqref{e:Pi}.
\end{itemize}
\end{lem}
\begin{proof}

 Take a copy $(\tl \Omega_1, \tl {\mcl F_1}, \tl \PP_1)$ of $(\Omega, \mcl F, \PP)$
and consider the following SPDEs on $(\tl \Omega_1, \tl{\mcl F}_1, \tl \PP_1)$:
\Be \label{e:SDE2x}
\begin{cases}
\dif X(t)=[A X(t)+F(X(t))] \dif t+\dif \tl L_t,   \\
X(0)=x,
\end{cases}
\Ee
and
\Be  \label{e:SDE2y}
\begin{cases}
\dif X(t)=[A X(t)+F(X(t))] \dif t+\dif \tl L_t,   \\
X(0)=y,
\end{cases}
\Ee
where $(\tl L_t)_{t \ge 0}$ has the same distribution as $(L_t)_{t \ge 0}$.

Define
\Be \label{e:TlTauDef}
\tl \tau=\inf\{t \ge T: |\Delta \tl L_t| \ge K\},
\Ee
it is easy to see that
$\tl \tau$ has the same distribution as $\tau$.
At the time $\tl \tau$, there is a jump $\tl \eta$ which is independent of
$\tl \tau$ and the processes $(X^x(t))_{0 \le t<\tl \tau}$ and
$(X^y(t))_{0 \le t<\tl \tau}$. We have
\Bes
X^x(\tl \tau)=X^x(\tl \tau-)+\tl \eta, \ \ X^y(\tl \tau)=X^y(\tl \tau-)+\tl \eta.
\Ees
Note that $\tl \eta$ is a random variable valued on ${\rm span}\{e_1,...,e_D\}$, then
\Bes
\begin{split}
&  X^x_1(\tl\tau)=X^x_1(\tl\tau-)+\tl \eta, \ \ X^y_1(\tl\tau)=X^y_1(\tl\tau-)+\tl \eta; \\
&  X^x_2(\tl\tau)=X^x_2(\tl\tau-), \ \ \ \ \ \ \ \ X^y_2(\tl\tau)=X^y_2(\tl\tau-). \\
\end{split}
\Ees
where $X_1$ and $X_2$ is defined according to the practice in \eqref{e:Pi}. For notational simplicity, write
\Be  \label{e:hatXY}
\begin{split}
& \hat x_1:=X^x_1(\tl \tau-), \ \ \hat y_1:=X^y_1(\tl\tau-); \\
& \hat x_2:=X^x_2(\tl \tau-), \ \ \hat y_2:=X^y_2(\tl\tau-).
\end{split}
\Ee
Note that $\hat x_1, \hat x_2, \hat y_1, \hat y_2$ above are all random variables on
$(\tl \Omega_1, \tl {\mcl F_1}, \tl \PP_1)$.

Now consider the conditional probabilities $\mcl L(X^x_1(\tl\tau) |\hat x_1)$ and
$\mcl L(X^y_1(\tl\tau) |\hat y_1)$, it is easy to see that these two probabilities
respectively have the following densities:
$$p_{K}(z-\hat x_1) \ \ \ \ \ {\rm and} \ \  \ \ p_{K}(z-\hat y_1).$$
By Theorem 4.2 of \cite{KS-2001}, there exists a probability space $(\tl \Omega_2, \tl{\mcl F}_2, \tl \PP_2)$
such that for any pair $(u, v) \in \R^D \times \R^D$, there exists a pair of random variables
$$\xi(u, v,\tl \omega_2 )=(\xi^x(u, v,\tl \omega_2 ),\xi^y(u, v,\tl \omega_2 ))$$ satisfying the following properties:
\begin{itemize}
\item[(i)]  $\xi(u, v)$ is a maximal coupling of
$p_K(z-u)$ and $p_K(z-v)$,
\item[(ii)] the map  $\xi(u, v,\tl \omega_2): \R^D \times \R^D \times \tl \Omega_2 \rightarrow \R^D\times \R^D$
is measurable.
\end{itemize}

\noindent Take $$(\tl \Omega, \tl{\mcl F}, \tl \PP)=(\tl \Omega_1, \tl{\mcl F}_1, \tl \PP_1)
\times (\tl \Omega_2, \tl{\mcl F}_2, \tl \PP_2),$$
from the procedure above,
$\xi(\hat x_1,\hat y_1,\tl \omega_2)$ is a maximal coupling
for the conditional probability $\mcl L(X^x_1(\tl\tau)|\hat x_1)$ and
$\mcl L(X^y_1(\tl\tau)|\hat y_1)$.
By the property of the maximal coupling,
\Be  \label{e:MaxCouPro}
\tl \PP \left(\xi^x(\hat x_1,\hat y_1) \ne \xi^y(\hat x_1,\hat y_1)|\hat x_1,\hat y_1\right)=\frac 12 \int_{\R^D} |p_K(z-\hat x_1)-p_K(z-\hat y_1)| \dif z.
\Ee
On $(\tl \Omega, \tl{\mcl F}, \tl \PP)$,
for every $\tl \omega \in \tl \Omega$ define
\Be
\begin{split}
& S^{x,y}(t,\tl \omega)=(X^x(t,\tl \omega_1),X^y(t,\tl \omega_1)), \ \ \ \ 0 \le t<\tl\tau(\tl \omega_1), \\
& S^{x,y}(\tl\tau(\tl \omega_1),\tl \omega)=(\xi^x(\hat x_1,\hat y_1)+\hat x_2, \xi^y(\hat x_1,\hat y_1)+\hat y_2),
\end{split}
\Ee
where
$$\xi^x(\hat x_1,\hat y_1)=\xi^x\big(\hat x_1(\tl\omega_1), \hat y_1(\tl \omega_1), \tl \omega_2\big),$$
$$\xi^y(\hat x_1,\hat y_1)=\xi^y\big(\hat x_1(\tl \omega_1), \hat y_1(\tl \omega_1), \tl \omega_2\big),$$
and $\hat x_1, \hat y_1, \hat x_2, \hat y_2$ are defined in \eqref{e:hatXY}. (3) follows \eqref{e:MaxCouPro}
immediately.

It remains to show (2).
From the above construction, it is clear that $(S^x(t))_{0 \le t<\tl \tau}$ and $(X^x (t))_{0 \le t < \tau}$
have the same distributions.
Since $X^x(\tau)=X^x(\tau-)+\eta$ with $\eta$ being independent of
$(X^x(t))_{0 \le t<\tau}$, given $X^x(\tau-)$, $X^x(\tau)$ is independent of $(X^x(t))_{0 \le t<\tau}$ and has
probability densities $p_K(z-X_1^x(\tau-))$ in the subspace $\HH_1$ and $\delta_{X_2^x(\tau-)}$ in the subspace $\HH_2$.
On the other hand,
from the above coupling construction, $\mcl L(S^x(\tl \tau)|S^{x,y}(\tl \tau-))$ has probability densities $p_K(z-S_1^x(\tau-))$ in the subspace $\HH_1$ and $\delta_{S_2^x(\tau-)}$ in the subspace $\HH_2$. Integrating over $S^{y}(\tl \tau-)$,
we obtain that $\mcl L(S^x(\tl \tau)|S^{x}(\tl \tau-))$ has probability densities $p_K(z-S_1^x(\tau-))$ in the subspace $\HH_1$ and $\delta_{S_2^x(\tau-)}$ in the subspace $\HH_2$.
This further implies that given $S^{x}(\tl \tau-)$, $\tl \eta:=S^x(\tl \tau)-S^x(\tl \tau-)$ has
a probability density $p_K(z)$ in the subspace $\HH_1$ and $\delta_0$ in the subspace $\HH_2$, it is clearly independent of $S^{x}(\tl \tau-)$.
Hence, $(X^x (t))_{0 \le t \le \tau}$ and $(S^x (t))_{0 \le t \le \tl \tau}$ have the same distributions.
By the same argument as above, we get that $(X^y (t))_{0 \le t \le \tl \tau}$ and $(S^y (t))_{0 \le t \le \tl \tau}$ have the same distributions.
\end{proof}

\begin{lem} \label{l:CouLem2}
Let $X^x(t)$ and $X^y(t)$ be the solution to the equation \eqref{e:MaiSPDE} for any give $x \in \HH$ and $y \in \HH$ respectively.
Then, there exists a probability space $(\bar \Omega, \bar {\mcl F}, \bar \PP)$ on which
\begin{itemize}
\item[(1)] there exists a Markov process $S^{x,y}(t)=(S^x(t),S^y(t))$ such that
$S^x(t)$ and $S^y(t)$ have the same distributions as those of $X^x(t)$ and $X^y(t)$ respectively;
\item[(2)] there exists a stopping times sequences $(\tl \tau_k)_{k \ge 0}$ which has the same
distribution as $(\tau_k)_{k \ge 0}$;
\item[(3)] the following equality holds: for all $k \ge 1$,
$$\bar \PP \bigg(S_1^x(\tl \tau_k) \ne S_1^y(\tl \tau_k)\big|S_1^x(\tl \tau_k-), S_1^y(\tl \tau_k-)\bigg)=\frac{1}{2} \int_{\R^D} |p_K(z-z_1)-p_K(z-z_2)| d z,$$
where $z_1=S_1^x(\tl \tau_k-)$ and $z_2=S_1^y(\tl \tau_k-)$.
\end{itemize}
\end{lem}

\begin{proof}
We shall prove the lemma by recursively applying Lemma \ref{l:CouLem1}. For the further use,
recall the notations in Lemma \ref{l:CouLem1} and
denote $\tl \tau_0:=0$, $\Delta \tl \tau_1:=\tl \tau$, $\tl \tau_1:=\tl \tau_0+\Delta \tl \tau_1$  and $(\tl \Omega^{1}, \tl {\mcl F}^{1}, \tl \PP^{1}):=(\tl \Omega, \tl {\mcl F}, \tl \PP)$.

Now taking $S^x(\tl \tau_1)$ and $S^y(\tl \tau_1)$ as initial data, by Lemma \ref{l:CouLem1} we have a probability space $(\tl \Omega^{(2)}, \tl {\mcl F}^{(2)}, \tl \PP^{(2)})$,
a copy of $(\tl \Omega, \tl {\mcl F}, \tl \PP)$ in Lemma \ref{l:CouLem1}, on which there exists a stopping time $\Delta \tl \tau_2$ and a Markov process $(S^{S^x(\tl \tau_1),S^y(\tl \tau_1)}(t))_{0 \le t \le \Delta \tl \tau_2}$ with the properties (1)-(3).

Denote $(\tl \Omega^{2}, \tl {\mcl F}^{2}, \tl \PP^{2})=(\tl \Omega^{1}, \tl {\mcl F}^{1}, \tl \PP^{1})\times(\tl \Omega^{(2)}, \tl {\mcl F}^{(2)}, \tl \PP^{(2)})$ and $\tl \tau_2=\tl \tau_1+\Delta \tl \tau_2$, on this new space, for every $\tl \omega=(\tl \omega_1, \tl \omega_2)\in(\tl \Omega^{2}, \tl {\mcl F}^{2}, \tl \PP^{2})$ define
$$S^{x,y}(t, \tl \omega)=S^{x,y}(t, \tl \omega_1), \ \ \ \ t \in [0, \tl \tau_1(\tl \omega_1)],$$
$$S^{x,y}(t, \tl \omega)=S^{S^x(\tl \tau_1(\tl \omega_1)),S^y(\tl \tau_1(\tl \omega_1))}(t-\tl \tau_1(\omega_1), \tl \omega_2), \ \ \ \ t \in [\tl \tau_1(\tl \omega_1), \tl \tau_2(\tl \omega_2)].$$

It is clear that $\tl \tau_2-\tl \tau_1$ has the same distribution as
$\tau$ and is independent of $\tl \tau_1$ and that $\tl \tau_2$ has the same distribution as $\tau_2$. We further claim that $(S^x(t))_{0 \le t \le \tl \tau_2}$ and $(S^y(t))_{0 \le t \le \tl \tau_2}$ have the same distributions as those of
$(X^x(t))_{0 \le t \le \tau_2}$ and $(X^y(t))_{0 \le t \le \tau_2}$ respectively. Indeed, if $X^x(\tau_1)=S^x(\tl \tau_1)=\hat x$ and $X^y(\tau_1)=S^y(\tl \tau_1)=\hat y$ with $\hat x, \hat y \in \HH$, by (2) of Lemma \ref{l:CouLem1}, we have
$$\mcl L((S^x(t))_{\tl \tau_1 \le t \le \tl \tau_2}|S^x(\tl \tau_1)=\hat x)=\mcl L((X^x(t))_{\tau_1 \le t \le \tau_2}|X^x(\tau_1)=\hat x).$$
From Lemma \ref{l:CouLem1}, $\mcl L(S^x(\tl \tau_1))=\mcl L(X^x(\tau_1))$. Hence, by Markov property, we have
$$\mcl L((S^x(t))_{0 \le t \le \tl \tau_2})=\mcl L((X^x(t))_{0 \le t \le \tau_2}).$$
Similarly,
$$\mcl L((S^y(t))_{0 \le t \le \tl \tau_2})=\mcl L((X^y(t))_{0 \le t \le \tau_2}).$$
By (3) of Lemma \ref{l:CouLem1}, the third property with $k=2$ in the lemma clearly holds.

Applying the same argument as above inductively, we have:
\begin{itemize}
\item[(i)] a sequence of probability spaces
$(\tl \Omega^{1}, \tl {\mcl F}^{1}, \tl \PP^{1}), (\tl \Omega^{2}, \tl {\mcl F}^{2}, \tl \PP^{2}), ... ,
(\tl \Omega^{k}, \tl {\mcl F}^{k}, \tl \PP^{k}),...$ with
$(\tl \Omega^{k}, \tl {\mcl F}^{k}, \tl \PP^{k})=(\tl \Omega, \tl {\mcl F}, \tl \PP) \times ...\times (\tl \Omega, \tl {\mcl F}, \tl \PP)$
being the $k$-tuple direct product of $(\tl \Omega, \tl {\mcl F}, \tl \PP)$ and
$(\bar \Omega, \bar {\mcl F}, \bar \PP)=(\tl \Omega^{\infty}, \tl {\mcl F}^{\infty}, \tl \PP^{\infty})$;
\item[(ii)] a sequence of stopping times $(\tl \tau_k)_{k \ge 1}$ such that $\tl \tau_k$ is located in $(\tl \Omega^{k}, \tl {\mcl F}^{k}, \tl \PP^{k})$ for each $k \ge 1$ and $(\tl \tau_k-\tl \tau_{k-1})_{k \ge 1}$ is i.i.d.
with the same distribution as $\tau$;
\item[(iii)] a sequence of coupling Markov processes
$((S^{x,y}(t))_{0 \le t \le \tl \tau_k})_{k\ge 1}$ such that $(S^{x,y}(t))_{0 \le t \le \tl \tau_k}$
is located in $(\tl \Omega^{k}, \tl {\mcl F}^{k}, \tl \PP^{k})$ and $\mcl L((S^x(t))_{0 \le t \le \tl \tau_k})=\mcl L((X^x(t))_{0 \le t \le \tau_k})$
and $\mcl L((S^y(t))_{0 \le t \le \tl \tau_k})=\mcl L((X^y(t))_{0 \le t \le \tau_k})$ for all $k \ge 1$. Moreover, (3) in the lemma holds.
\end{itemize}
It is clear from (ii) that $\lim_{k \rightarrow \infty}\tau_k=\infty$ a.s. and
$\lim_{k \rightarrow \infty}\tl \tau_k=\infty$ a.s., this, together with (iii), immediately implies
(1) in the lemma.
\end{proof}

\subsection{Some estimates of the coupling chain $(S^{x,y}(\tl \tau_k))_{k \ge 0}$}
Recall that
$(S^{x,y}(\tl \tau_k))_{k \ge 0}$ is a Markov chain on the probability space
$(\bar \Omega, \bar {\mcl F}, \bar {\PP})$.
Note that $(\bar \Omega, \bar {\mcl F}, \bar {\PP})$ is not necessarily
the same as $(\Omega, {\mcl F}, {\PP})$ on which $(X^x(t))_{t \ge 0}$ and $(X^y(t))_{t \ge 0}$ is located.
Without loss of generality, we assume that
\Be \label{a:SpaAss}
(\Omega, {\mcl F}, {\PP})=(\bar \Omega, \bar {\mcl F}, \bar {\PP}).
\Ee
Otherwise we can introduce the product space $(\bar \Omega \times \Omega, \bar {\mcl F} \times \mcl F, \bar {\PP} \times \PP)$
and consider $(S^{x,y}(k))_{k \ge 0}$, $(X^x(t))_{t \ge 0}$ and $(X^y(t))_{t \ge 0}$ all together on this new space. However,
this will make the notations unnecessarily complicated, for instance, we have to always use $\bar \PP \times \PP$.

\begin{prop}   \label{p:SquKOrd}
Let  $\{\tl \tau_k\}_{k \ge 0}$ be the stopping times sequence in Lemma \ref{l:CouLem2} and let
\Be \label{e:Kappa}
\delta_k:=e^{-\lambda_{D+1} (\tl\tau_{k+1}-\tl\tau_k)}+\frac{\|F\|_{\rm Lip} e^{(\tl\tau_{k+1}-\tl\tau_k) \|F\|_{\rm Lip}}}{\lambda_{D+1}+\|F\|_{\rm Lip}}.
\Ee
For all $x, y \in \HH$,
if $\|S^x(\tl \tau_k)\|_{\HH}+\|S^y(\tl \tau_k)\|_{\HH} \le M$ with $M$ being the number in Assumption \ref{a:}, we have
\Bes
\PP\left\{\|S^x(\tl \tau_{k+1})-S^y(\tl \tau_{k+1})\|_{\HH}>\delta_k \|S^x(\tl \tau_{k})-S^y(\tl \tau_{k})\|_{\HH} \big |S^{x,y}(\tl \tau_k)\right\} \le {\beta_0}/{2},
\Ees
with $\beta_0$ being the constant in Assumption \ref{a:}. Furthermore, we have
\Bes 
\begin{split}
\PP\big\{\|S^x(\tl \tau_{k+1})-S^y(\tl \tau_{k+1})\|_{\HH} \ge \delta_k \|S^x(\tl \tau_{k})-S^y(\tl \tau_{k})\|_{\HH} \big |S^{x,y}(\tl \tau_{k})\big\}  \le \kappa \|S^x(\tl \tau_{k})-S^y(\tl \tau_{k})\|_{\HH}^{\beta_2}
\end{split}
\Ees
for all $k \ge 0$,
where $\kappa=\beta_1 e^{\beta_2 \|F\|_{\rm Lip} T}$ and
$\beta_1, \beta_2$ are the constants in Assumption \ref{a:}.
\end{prop}

\begin{proof}
The proofs of the both inequalities are similar, we only show the second one, which is more difficult than the first.
Since $\{S^{x,y}(\tl \tau_k)\}_{k \ge 0}$ is a time-homogeneous Markov chain,
it suffices to show the inequality for $k=0$, i.e.
\Be \label{e:Squ1Ord}
\PP\left(\|S^x(\tl \tau_1)-S^y(\tl \tau_1)\|_{\HH} \ge \delta_0 \|x-y\|_{\HH}\right) \le \kappa \|x-y\|_{\HH}^{\beta_2}.
\Ee
By the construction of the coupling process $\{S^{x,y}(t)\}_{0 \le t \le \tl \tau_1}$ in Lemma \ref{l:CouLem1}, $S^{x,y}(\tl \tau_1)$ have
\Be \label{e:BarxNeqBary}
\begin{split}
& \ \ \PP \left(\|S^x(\tl \tau_1)-S^y(\tl \tau_1)\|_{\HH}>\delta_0 \|x-y\|_{\HH}\right) \\
& \le \PP \left(S_1^x(\tl \tau_1) \ne S_1^y(\tl \tau_1)\right)+\PP \left(S_1^x(\tl \tau_1)=S_1^y(\tl \tau_1), \ \|S_2^x(\tl \tau_1)-S_2^y(\tl \tau_1)\|_{\HH}>\delta_0 \|x-y\|_{\HH}\right).
\end{split}
\Ee

\noindent On the one hand, it follows from (3) of Lemma \ref{l:CouLem1} that
\Bes
\begin{split}
 \PP \left(S_1^x(\tl \tau_1) \ne S_1^y(\tl \tau_1)\right)& =\E \left[\PP\big(S_1^x(\tl \tau_1) \ne S_1^y(\tl \tau_1) \big |(S^x(\tl \tau_1-), S^y(\tl \tau_1-))\big)\right] \\
 & =  \frac{1}{2} \E \int_{\R^D} |p_K(z-S_1^x(\tl \tau_1-))-p_K(z-S_1^y(\tl \tau_1-))| d z \\
&  \le  \frac{\beta_1}{2}\E \left|S_1^x(\tl \tau_1-)-S_1^y(\tl \tau_1-)\right|^{\beta_2},
\end{split}
\Ees
where  the inequality is by (A2) of Assumption \ref{a:}.

This, together with (2) of Lemma \ref{l:SolEst} and (A4) of Assumption \ref{a:}, imply
\Be
\begin{split}
\PP \left(S_1^x(\tl \tau_1) \ne S_1^y(\tl \tau_1)\right)  &\le \frac{\beta_1}{2} \E\left[e^{\beta_2 \|F\|_{\rm Lip} \tl \tau_1}\right] \|x-y\|_{\HH}^{\beta_2}\\
& \le \kappa \|x-y\|_{\HH}^{\beta_2}.
\end{split}
\Ee
On the other hand, it follows from (2) of Lemma \ref{l:SolEst} that
$$\|S^x_2(\tl \tau_1)-S^y_2(\tl \tau_1)\|_{\HH} \le \delta_0 \|x-y\|_{\HH} \ \ \ a.s., $$
therefore,
\Be \label{e:Xx=Xy}
\PP \left(\|S^x_2(\tl \tau_1-)-S^y_2(\tl \tau_1-)\|_{\HH}>\delta_0 \|x-y\|_{\HH}\right)=0.
\Ee
Collecting \eqref{e:BarxNeqBary}-\eqref{e:Xx=Xy}, we immediately get the desired inequality.
\end{proof}
\ \ \

\section{Proof of main theorem}
For notational simplicity, we shall simply write the coupling chain as
$$S^{x,y}(k)=S^{x,y}(\tl \tau_k), \ \ \ \ \ \ k \ge 0,$$
and drop the superscript $x,y$ whenever no confusions arise.
Let us briefly give the strategy of the proof as below:
\begin{itemize}
\item [(i)] We first estimate
\Bes
|\E[f(X^x(\tau_k))]-\E[f(X^y(\tau_k))]|
\Ees
for any $f \in L_b(\HH)$ and any $x, y \in \HH$, and then compare the $X(t)$ and
$X(\tau_k)$ to get the exponential mixing of $X(t)$.
\item [(ii)] By the coupling we have constructed in previous section, we have
\Bes
\E[f(X^x(\tau_k))]-\E[f(X^y(\tau_k))]=\E[f(S^x(k))]-\E[f(S^y(k))].
\Ees
\item [(iii)] To estimate $|\E[f(S^x(k))]-\E[f(S^y(k))]|$ for sufficiently large $k$, we need to introduce some stopping times and estimate them.
Roughly speaking, these stopping times can be simplified as
\Bes
\begin{split}
& \tl \sigma=\inf\{k>0; \| S^x(k)\|_{\HH}+\| S^y(k)\|_{\HH} \le M\}, \\
& \hat \sigma=\inf\left\{k>0; \| S^x(k)- S^y(k)\|_{\HH} \ge \frac{\tl \tau_k}{\lambda^k_{D+1}}\right\},
\end{split}
\Ees
the $\hat \sigma$ is exactly defined in \eqref{d:SigXY}, but the above
definition captures the essential part of \eqref{d:SigXY}.
We show that $\E\exp\left(\theta \tl \sigma\right)<\infty$ for some $\theta>0$ and $\PP(\hat \sigma=\infty)>0$.
The two relations roughly mean that the system $( S(k))_{k \ge 0}$ enters the $M$-radius ball
exponentially frequently, for some sample paths (with positive probability) in the ball, $\| S^x(k)-S^y(k)\|_{\HH}$ converges to zero exponentially fast
as long as $\lambda_{D+1}$ is sufficiently large.
\end{itemize}
\subsection{Some estimates of stopping times of the coupling chain $(S^{x,y}(k))_{k \ge 0}$}
In this subsection, we shall construct stopping times of the coupling chain $(S^{x,y}(k))_{k \ge 0}$
and give some auxiliary lemmas for proving the main theorem. The proofs of these auxiliary lemmas can
be found in \cite{Xu14}.

Given $M, d>0$, $x,y\in\HH$, define the stopping times
\Be \label{d:SigM}
\tl \sigma(x,y,M):=\inf \left\{k>0; \| S^x(k)\|_{\HH}+\| S^y(k)\|_{\HH} \le M\right\},
\Ee
\Be \label{d:SigD}
\sigma(x,y,d):=\inf \left\{k>0; \| S^x(k)- S^y(k)\|_{\HH} \le d\right\},
\Ee
we set $\tl \sigma=\tl \sigma(x,y,M)$, $\sigma=\sigma(x,y,d)$ in shorthand if no confusions arise. Let us prove the following
two theorems:

\begin{thm} \label{l:TauPro}
For all $p \in (0, \alpha)$, as $T>T_0:=\frac{(p-1)\log 3}{p \lambda_1} \vee 0$ with $T$ being number defined in \eqref{stopping time}, there
exist positive constants $M, \tl \vartheta, C$ depending on $p,\lambda_1, \|F\|_0, T$ so that
\Bes
\E_{(x,y)} [e^{\tl \vartheta \tl \sigma}]<C(1+\|x\|_{\HH}^p+\|y\|_{\HH}^p)
\Ees
for all $x, y \in \HH$.
\end{thm}
\begin{proof}
See Theorem 4.1 in \cite{Xu14}.
\end{proof}

\begin{thm} \label{t:SigDEst}
As $D$ is sufficiently large, there exists some constant $\vartheta>0$ such that for all $p \in (0, \alpha)$, $d>0$ and $x, y \in \HH$,
\Be \label{e:SigEst}
\E_{(x,y)} [e^{\vartheta \sigma}]\leq C\bg(1+\|x\|_{\HH}^p+\|y\|_{\HH}^p\bg)
\Ee
where $C$ depends on $p,  \lambda_1, \lambda_{D+1}, \|F\|_{0}, \|F\|_{\rm Lip}, \vartheta, d, K$.
\end{thm}
\begin{proof}
See Theorem 4.2 in \cite{Xu14}.
\end{proof}

\ \ \ \ \ \ \

Define
\Be \label{d:SigXY}
\hat \sigma (x,y):=\inf\{k \ge 1; \|S^x(k)-S^y(k)\|_{\HH}>(\delta_0 \dots \delta_{k-1}) \|x-y\|_{\HH}\},
\Ee
where $\delta_j$ ($j=0, ..., k-1$) are defined in Proposition \ref{p:SquKOrd}, we shall often
write $\hat \sigma=\hat \sigma(x,y)$ in shorthand.

\begin{lem} \label{l:SigInf}
If $\|x-y\|_{\HH} \le d$ with $0<d<\left(\frac 1{4 \kappa}\right)^{1/\beta_2}$ and $\kappa$ defined in
Proposition \ref{p:SquKOrd}, as $D$ is large enough, we have
\begin{enumerate}
\item $\PP_{(x,y)} (\hat \sigma=\infty)>1/2.$
\item There exists some $\epsilon>0$ (possibly small) depending on
$d,\lambda_1, \lambda_{D+1}, \|F\|_0, \|F\|_{\rm Lip}, \alpha, K, \epsilon$ such that
$$\E_{(x,y)} [e^{\epsilon \hat \sigma} 1_{\{\hat \sigma<\infty\}}] \le C,$$
where $C$ depends on $d, \lambda_1, \lambda_{D+1},\|F\|_0, \|F\|_{\rm Lip}, \alpha, K, \epsilon$.
\end{enumerate}
\end{lem}
\begin{proof}
See Lemma 5.1 in \cite{Xu14}.
\end{proof}

Define
\Be \label{d:SigDag}
\sigma^{\dag}(x,y,d):=\sigma+\hat \sigma(S^{x,y}(\sigma)),
\Ee
where $\sigma=\sigma(x,y,d)$ defined by \eqref{d:SigD}. Further define
\Be \label{d:BarSig}
\bar \sigma(x,y,d,M):=\sigma^\dag+\tl \sigma(S^{x,y}(\sigma^\dag),M),
\Ee
where $\sigma^\dag=\sigma^\dag(x,y,d)$ and $\tl \sigma$ is defined in \eqref{d:SigM}.

The motivation for defining $\bar \sigma$ is the following: we only know
$\|S^x(\sigma^\dag)-S^y(\sigma^\dag)\|_{\HH} \le d$, but have no idea about
the bound of $\|S^x(\sigma^\dag)\|_{\HH}+\|S^y(\sigma^\dag)\|_{\HH}$.
This bound is very important for iterating a stopping time argument as in Step 1 of
the proof of Theorem \ref{t:SigDEst}. To this aim, we introduce \eqref{d:BarSig}
and thus have
\Be \label{e:SBarSigM}
\|S^{x,y}(\bar \sigma)\|_{\HH} \le M \ \ \ \forall \ x,y \in \HH.
\Ee
\begin{lem} \label{l:BarSigEst}
Let $0<d<\left(\frac 1{4 \kappa}\right)^{1/\beta_2}$ and $p \in (0, \alpha)$. There exist some $\gamma, C>0$
depending on $d,\lambda_1, \lambda_{D+1}, \|F\|_0, \|F\|_1, p, M, K$ such that
\Bes
\E_{(x,y)}[e^{\gamma \bar \sigma (x,y,d,M)} 1_{\{\bar \sigma(x,y,d,M)<\infty\}}] \le C(1+\|x\|_{\HH}^p+\|y\|_{\HH}^p).
\Ees
\end{lem}
\begin{proof}
See Lemma 5.2 in \cite{Xu14}.
\end{proof}

Define $\bar \sigma_0=0$, for all $k \ge 0$ we define
$$\bar \sigma_{k+1}=\bar \sigma_k+\bar \sigma(S^{x,y}(\bar \sigma_k), d, M), $$
it is easy to see that each $\bar \sigma_k$ depends on $x, y, d, M$.

\begin{lem} \label{l:BarSigK}
Let $k \in \N$. For all $x,y \in \HH$, we have
\Be
\PP_{(x,y)} \left(\bar \sigma_k<\infty\right) \le 1/2^k.
\Ee
\end{lem}
\begin{proof}
See Lemma 5.3 in \cite{Xu14}.
\end{proof}

\subsection{Proof of the main theorem} \label{s:PfMaiThm}
\begin{proof} [{\bf Proof of Theorem \ref{t:MaiThm}}]
The existence of invariant measures has been established in \cite{PXZ10}.
According to \cite[Sect. 2.2.]{Shi08}, the inequality \eqref{2.5} in the theorem implies
the uniqueness of the invariant measure. So now we only need to show \eqref{2.5}, by
\cite{Shi08} again,
it suffices to show that for all $p \in (0, \alpha)$ we have
\Be \label{e:EquInqMai}
 |P_tf(x)-P_t f(y)| \leq C e^{-ct}\|f\|_1 (1+\|x\|_{\HH}^{p}+\|y\|_{\HH}^p) \ \ \ \forall \ f \in L_b(\mathbb H),
\Ee
where $C,c$ depend on $p,\beta,K, \|F\|_{1}, \lambda_1, \lambda_{D+1}$.

Define $g(x,t):=\E f(X^x_t)$, it is easy to see
\Be \label{e:BouG}
\|g\|_0 \le \|f\|_0,
\Ee
 by the
third inequality in Lemma \ref{l:SolEst} we further have
\
\Be \label{e:LipG}
\begin{split}
|g(x, t)-g(y, t)| &\le \E \left[\left|f(X^{x}(t))-f(X^{y}(t))\right|\right] \\
& \le \|f\|_1 \E \left[\|X^{x}(t)-X^{y}(t)\|_{\HH}\right]  \\
& \le \|f\|_1 e^{t \|F\|_{\rm Lip}}\|x-y\|_{\HH}.
\end{split}
\Ee
By strong Markov property, on the set $\{\tau_l \le t\}$ we have
$$g(u_x, t-\tau_l)=\E[f(X^x(t))|X^x(\tau_l)], \ \ \ \ g(u_y, t-\tau_l)=\E[f(X^y(t))|X^y(\tau_l)],$$
where $u_x=X^x(\tau_l), u_y=X^y(\tau_l)$.

To prove \eqref{e:EquInqMai}, we claim and prove below that for all $l \in \N$ and $f \in L_b(\mathbb H)$
\Be \label{e:CouFEst}
\begin{split}
& \left|\E \left[g(u_x, t-\tau_l)-g(u_y, t-\tau_l)\right]1_{\{\tau_l\leq t\}}\right| \\
 \le & \|f\|_0 2^{-\e l+1}+C\|f\|_0 (e^{-\vartheta l/8}+e^{-\frac{\gamma l}4}) (1+\|x\|_{\HH}^p+\|y\|_{\HH}^p)+\e d l\|f\|_1 \bigg(\frac 2{\lambda_{D+1}}\bigg)^{l/4}.
\end{split}
\Ee
Let $l \ge 2$ be some natural number to be determined later. Then, we easily have
\
\Bes
\begin{split}
|P_tf(x)-P_t f(y)|&=\left|\E[f(X^x(t))]-\E[f(X^y(t))]\right| \\
& \le \left|\E \left\{\left[f(X^x(t))-f(X^y(t))\right]1_{\left\{\tau_{l}>t\right\}} \right\}\right| \\
& \ \ +\left|\E \left\{\left[f(X^x(t))-f(X^y(t))\right]1_{\left\{\tau_{l} \le t\right\}} \right\}\right| \\
& \le 2 \|f\|_0 \PP\left(\tau_{l}>t\right)+\left|\E\left\{\left[g(X^x(\tau_l), t-\tau_l)-g(X^y(\tau_l), t-\tau_l)\right]1_{\left\{ \tau_{l} \le t\right\}}\right\}\right| \\
& \le  2 \|f\|_0 \left(2 e^{\gamma_K T/2}\right)^{l} e^{-\gamma_K t/2}+ \|f\|_0 2^{-\e l+1}+C\|f\|_0 (e^{-\vartheta l/8}+e^{-\frac{\gamma l}4}) (1+\|x\|_{\HH}^p+\|y\|_{\HH}^p) \\
&\ \ \ +\e d l\|f\|_1 e^{t \|F\|_{\rm Lip}}\bigg(\frac 2{\lambda_{D+1}}\bigg)^{l/4},
\end{split}
\Ees
where the last inequality is by \eqref{e:CouFEst} and the following easy estimate
$$\PP(\tau_{l}>t) \le \left(2 e^{\gamma_K T/2}\right)^{l} e^{-\gamma_K t/2}.$$
Choosing $l=[\delta t]$ with $\delta>0$ sufficiently small (depending on $p,\lambda_1,\|F\|_{1}, K, \beta, M$) and then choosing $D$ sufficiently large (depending on $p, \lambda_1, \|F\|_{1}, K, \beta, M, \delta$) so that $\lambda_{D+1}$ sufficiently large, we immediately get
$$|P_tf(x)-P_t f(y)| \le C_0 e^{-c_0 t} \|f\|_1(1+\|x\|_{\HH}^p+\|y\|_{\HH}^p),$$
where $C_0$ and $c_0$ both depending on $p, \lambda_1, \lambda_{D+1},\|F\|_{1}, K, \beta, M$.

\vskip 3mm
Now it remains to show \eqref{e:CouFEst}.
Let $m=[\e l]$ with $0<\e<1/2$ to be determined later. By the coupling in Lemma \ref{l:CouLem2}, we have
\Bes
\begin{split}
& \ \left|\E \left[g(X^x(\tau_l), t-\tau_l)-g(X^y(\tau_l), t-\tau_l)\right]1_{\{\tau_l\leq t\}}\right| \\
=&\left|\E \left[g(S^x(l), t-\tl \tau_l)-g(S^y(l), t-\tl \tau_l)\right]1_{\{\tl \tau_l\leq t\}}\right| \leq J_{1}+J_{2},
\end{split}
\Ees
where we write $S^x(l)=S^x(\tl \tau_l)$ for simplicity and
\Bes
\begin{split}
 & J_{1}=\left|\E \left\{ \left(g( S^x(l), t-\tl \tau_l)-g(S^y(l), t-\tl \tau_l)\right)1_{\left\{\tl \tau_{l} \le t\right\}} 1_{\{\bar \sigma_m<\infty\}}\right\}\right|,\\ & J_{2}=\left|\E \left\{ \left(g(S^x(l), t-\tl \tau_l)-g(S^y(l), t-\tl \tau_l)\right)1_{\left\{\tl \tau_{l} \le t\right\}} 1_{\{\bar \sigma_m=\infty\}}\right\}\right|.
\end{split}
\Ees
By Lemma \ref{l:BarSigK}, we have
\Be
J_{1} \le 2 \|f\|_0 \PP_{(x,y)}\{\bar \sigma_m<\infty\}\le \frac {\|f\|_0}{2^{m-1}} \le \|f\|_0 2^{-\e l+1}.
\Ee
Observe
\Bes
J_2=J_{2,1}+J_{2,2},
\Ees
where
\Bes
\begin{split}
&J_{2,1}=\sum_{i=0}^{m-1}\E \left\{\left|g(S^x(l), t-\tl \tau_l)-g(S^y(l), t-\tl \tau_l)\right|1_{\left\{\tl \tau_{l} \le t\right\}}1_{\{l/2<\bar \sigma_i<\infty, \ \bar \sigma_{i+1}=\infty\}}\right\}, \\
&J_{2,2}=\sum_{i=0}^{m-1}\E \left\{\left|g(S^x(l), t-\tl \tau_l)-g(S^y(l), t-\tl \tau_l)\right|1_{\left\{\tl \tau_{l} \le t\right\}} 1_{\{\bar \sigma_i \le l/2, \ \bar \sigma_{i+1}=\infty\}}\right\}.
\end{split}
\Ees
By Chebyshev inequality and Lemma \ref{l:BarSigEst},
\Bes
\begin{split}
J_{2,1}  & \le 2\|f\|_0 \sum_{i=0}^{m-1}\PP_{(x,y)}\left\{\frac l2<\bar \sigma_i<\infty\right\}
\le 2\|f\|_0 e^{-\frac l 2 \gamma} \sum_{i=0}^{m-1}   \E_{(x,y)}[e^{\gamma \bar \sigma_i}]
\end{split}
\Ees
and
\Bes
\begin{split}
\E_{(x,y)}[e^{\gamma \bar \sigma_i}]&=\E_{(x,y)} \left[e^{\gamma \bar \sigma_1}\E_{S(\bar \sigma_1)}
\left[e^{\gamma (\bar \sigma_2-\bar \sigma_1)}\cdots \E_{S(\bar \sigma_{i-1})}\left[e^{\gamma (\bar \sigma_i-\bar \sigma_{i-1})}\right]\cdots \right]\right]\\
& \leq C^i e^{\gamma i} (1+M^{p})^{i-1} (1+\|x\|_{\HH}^{p}+\|y\|_{\HH}^p),
\end{split}
\Ees
where the last inequality is by \eqref{e:SBarSigM}.
Hence,
$$J_{2,1} \le 2\|f\|_0 e^{-\frac l 2 \gamma} \sum_{i=0}^{m-1} e^{\gamma i} (1+M^{p})^{i-1} (1+\|x\|_{\HH}^{p}+\|y\|_{\HH}^p).$$
Recall $m=[\e l]$, as $\e>0$ is sufficiently small we obtain
\Bes
J_{2,1} \le e^{-\frac{\gamma l}4} \|f\|_0 (1+\|x\|_{\HH}^p+\|y\|_{\HH}^p).
\Ees
\vskip 2mm
It remains to estimate $J_{2,2}$. Recall the definition of
$\sigma, \hat \sigma, \sigma^\dag, \bar \sigma, \tl \sigma$ and note that
\Be \label{e:BarSigMore}
\bar \sigma_{i+1}=\bar \sigma_i+\sigma+\hat \sigma+\tl \sigma,
\Ee
with
$\sigma=\sigma(S^{x,y}(\bar \sigma_i),d)$,
$\hat \sigma=\hat \sigma(S^{x,y}(\bar \sigma_i+\sigma))$, $\tl \sigma=\tl \sigma(S^{x,y}(\bar \sigma_i+\sigma+\hat \sigma),M)$.
Observe that
\Bes
J_{2,2}=J_{2,2,1}+J_{2,2,2},
\Ees
with
\Bes
\begin{split}
& J_{2,2,1}:=\sum_{i=0}^{m-1}\E \left[\left|g(S^x(l), t-\tl \tau_l)-g(S^y(l), t-\tl \tau_l)\right|1_{\left\{\tl \tau_{l} \le t\right\}} 1_{\{\bar \sigma_i \le l/2, \ \bar \sigma_i+\sigma>\frac{3l}4, \bar \sigma_{i+1}=\infty\}}\right], \\
&J_{2,2,2}:=\sum_{i=0}^{m-1} \E \left[\left|g(S^x(l), t-\tl \tau_l)-g(S^y(l), t-\tl \tau_l)\right|1_{\left\{\tl \tau_{l} \le t\right\}} 1_{\{\bar \sigma_i \le l/2, \bar \sigma_i+\sigma \le \frac{3l}4, \bar \sigma_{i+1}=\infty\}}\right].
\end{split}
\Ees
By strong Markov property, Chebyshev inequality, Theorem \ref{t:SigDEst} and the clear fact $\|S^{x,y}(\bar \sigma_i)\|_{\HH}\le M$ for all $i \ge 1$, as $\e>0$
is sufficiently small we have
\Be
\begin{split}
J_{2,2,1} & \le 2 \|f\|_0 \sum_{i=0}^{m-1} \E_{(x,y)}  \left[\PP_{u_i}(\sigma>l/4) \right] \\
& \le  C  \|f\|_0 e^{-\vartheta l/4}\big[(m-1)(1+M^p)+(1+\|x\|_{\HH}^p+\|y\|_{\HH}^p)\big] \\
& \le C\|f\|_0 e^{-\vartheta l/8} (1+\|x\|_{\HH}^p+\|y\|_{\HH}^p),
\end{split}
\Ee
where $u_i=S^{x,y}(\bar \sigma_i)$ and $C, \vartheta$ depend on  $d,\lambda_1, \lambda_{D+1},\|F\|_1, p, M$.
\vskip 2mm

As for $J_{2,2,2}$, recall \eqref{e:BarSigMore} and note $\tl \sigma<\infty$ a.s. from Theorem \ref{l:TauPro}, we have
\Be
\begin{split}
J_{2,2,2}
&=\sum_{i=0}^{m-1} \E \left\{\left|g(S^x(l), t-\tl \tau_l)-g(S^y(l), t-\tl \tau_l)\right|1_{\left\{\tl \tau_{l} \le t\right\}} 1_{\{\bar \sigma_i \le l/2, \bar \sigma_i+\sigma \le \frac{3l}4, \hat \sigma+\tl \sigma=\infty\}}\right\} \\
&=\sum_{i=0}^{m-1} \E \left\{\left|g(S^x(l), t-\tl \tau_l)-g(S^y(l), t-\tl \tau_l)\right|1_{\left\{\tl \tau_{l} \le t\right\}} 1_{\{\bar \sigma_i \le l/2, \bar \sigma_i+\sigma \le \frac{3l}4, \hat \sigma=\infty\}}\right\}.
\end{split}
\Ee
 It follows from the above equality, \eqref{e:LipG} and strong Markov property that
\Bes
\begin{split}
J_{2,2,2} & \le \|f\|_1 e^{t \|F\|_{\rm Lip}} \sum_{i=0}^{m-1} \E\left[\|S^x(l)-S^y(l)\|_{\HH} 1_{\left\{\tl \tau_{l} \le t\right\}}
1_{\{\bar \sigma_i+\sigma \le \frac{3l}4, \hat \sigma=\infty\}}\right] \\
&=\|f\|_1 e^{t \|F\|_{\rm Lip}} \sum_{i=0}^{m-1} \E\left[\E_{u}\left(\|S^x(l)-S^y(l)\|_{\HH} 1_{\left\{\tl \tau_{l} \le t\right\}} 1_{\{\hat \sigma=\infty\}}\right)
1_{\{\bar \sigma_i+\sigma \le \frac{3l}4\}}\right],
\end{split}
\Ees
where $u=S^{x,y}(\bar \sigma_i+\sigma)$. By the definition of $\sigma$ we have $\|u_x-u_y\|_{\HH}<d$. By the definition \eqref{d:SigXY} with $\hat \sigma=\hat \sigma(S^{x,y}(\bar \sigma_i+\sigma))$ and the previous inequality, as $\lambda_{D+1}>0$ is sufficiently large, depending on $T,K,\|F\|_{\rm Lip}$, we have
\Bes
\begin{split}
J_{2,2,2}  & \le \|f\|_1  \sum_{i=0}^{m-1} \E\left[\E_u(\delta_0 ... \delta_{l/4})\|u_x-u_y\|_{\HH}\right] \\
& \le  d m\|f\|_1 \bigg(\frac{\gamma_K}{\gamma_K+\lambda_{D+1}}e^{-\lambda_{D+1} T}+2 \frac{\|F\|_{\rm Lip}}{\|F\|_{\rm Lip}+\lambda_{D+1}}\bigg)^{l/4} \\
& \le \e d l\|f\|_1 \bigg(\frac 2{\lambda_{D+1}}\bigg)^{l/4}.
\end{split}
\Ees

\vskip 2mm
\noindent Collecting the bounds for $J_{2,2,1}, J_{2,2,2}$, $J_{2,1}$, $J_{1}$, we have that there exist some $\epsilon, C>0$ depending on
$p, \lambda_1, \lambda_{D+1},\|F\|_{1}, K$ such that
\Bes
\begin{split}
&\left|\E \left[g(X^x(\tau_l), t-\tau_l)-g(X^y(\tau_l), t-\tau_l)\right]1_{\{\tau_l\leq t\}}\right|\\
 \le & J_1+J_{2,1}+J_{2,2,1}+J_{2,2,2}\\
 \le & \|f\|_0 2^{-\e l+1}+C\|f\|_0 (e^{-\vartheta l/8}+e^{-\frac{\gamma l}4}) (1+\|x\|_{\HH}^p+\|y\|_{\HH}^p)+\e d l\|f\|_1 \bigg(\frac 2{\lambda_{D+1}}\bigg)^{l/4},
\end{split}
\Ees
this proves the desired \eqref{e:CouFEst}.
\end{proof}


\bibliographystyle{amsplain}

\end{document}